\documentclass[11pt]{amsart}
\usepackage{amsmath,amssymb,amsthm}
\usepackage{hyperref}
\hypersetup{hidelinks=true}

\usepackage{bbm}

\usepackage[utf8]{inputenc}

\usepackage{graphicx}
\usepackage{enumerate}
\usepackage{amsfonts,amssymb,latexsym}
\usepackage{accents}
\usepackage{afterpage}
\usepackage[style=alphabetic,maxbibnames=99, backend=bibtex]{biblatex}
\usepackage{tikz}

\newcommand{\ep}{\varepsilon}
\newcommand{\N}{\mathbb{N}}

\newcommand{\R}{\mathbb{R}}

\newcommand{\Z}{\mathbb{Z}}

\newcommand\restr[2]{{
  \left.\kern-\nulldelimiterspace 
  #1 
  \vphantom{\big|} 
  \right|_{#2} 
  }}

\DeclareMathOperator{\co}{co}

\renewcommand{\>}{\rangle}

\newcommand{\cco}{\overline{\operatorname{co}}}

\newtheorem{fact}{Fact}[section]

\newtheorem{theorem}{Theorem}[section]
\newtheorem{lemma}[theorem]{Lemma}
\newtheorem{claim}[theorem]{Claim}
\newtheorem{prop}[theorem]{Proposition}
\newtheorem{corollary}[theorem]{Corollary}
\theoremstyle{definition}
\newtheorem{definition}[theorem]{Definition}
\newtheorem{example}[theorem]{Example}

\theoremstyle{remark}
\newtheorem{remark}[theorem]{Remark}
\numberwithin{equation}{section}

\def\fnote#1{\footnote}

\def\R{{\mathbb R}}

\def\ignora#1{}
\def\n3#1{\left\vert  \! \left\vert \! \left\vert \, #1 \, \right\vert \!
  \right\vert \! \right\vert }

\include{draft}

\begin{document}



\title[Hölder retractions]{Compact Hölder retractions and nearest point maps}


\author{ Rub\'en Medina}\thanks{This research was supported by CAAS CZ.02.1.01/0.0/0.0/16-019/0000778, project  SGS21/056/OHK3/1T/13, MICINN (Spain) Project PGC2018-093794-B-I00 and MIU (Spain) FPU19/04085 Grant.}
\address[R. Medina]{Universidad de Granada, Facultad de Ciencias.
Departamento de An\'{a}lisis Matem\'{a}tico, 18071-Granada
(Spain); and Czech Technical University in Prague, Faculty of Electrical Engineering.
Department of Mathematics, Technická 2, 166 27 Praha 6 (Czech Republic)}
\email{rubenmedina@ugr.es}
\urladdr{\url{https://www.ugr.es/personal/ae3750ed9865e58ab7ad9e11e37f72f4}}

\maketitle

\begin{abstract}
In this paper, two main results concerning uniformly continuous retractions are proved. First, an $\alpha$-Hölder retraction from any separable Banach space onto a compact convex subset whose closed linear span is the whole space is constructed for every positive $\alpha<1$. This constitutes a positive solution to a Hölder version of a question raised by Godefroy and Ozawa. In fact, compact convex sets are found to be absolute $\alpha$-Hölder retracts under certain assumption of flatness. Second, we provide an example of a strictly convex Banach space $X$ arbitrarily close to $\ell_2$ (for the Banach Mazur distance) and a finite dimensional compact convex subset of $X$ for which the nearest point map is not uniformly continuous even when restricted to bounded sets.
\end{abstract}

\section{Introduction}\label{mainsect}
In this note we study retractions onto compact convex subsets of Banach spaces. A retraction is the natural nonlinear counterpart to the concept of projection for subsets of Banach spaces. In particular, Lipschitz and uniformly continuous retractions have been widely studied (see \cite{Phe58}, \cite{Lin64}, \cite{Bjo79}), often in the search of a way to extend Lipschitz and uniformly continuous maps (see \cite{Whi34} ,\cite{JLS86}, \cite{Bou87}, \cite{LN05}, \cite{God215}, \cite{BDS21}).

The first main result in this paper is motivated by the natural question asked by Godefroy and Ozawa in \cite{GO14}, and then subsequently in \cite{God15}, \cite{God215}, \cite{GMZ16}, \cite{GP19} and \cite{God20}. They wonder whether for every separable Banach space there is a Lipschitz retraction onto a compact convex subset whose closed linear span is the whole space. In \cite{HM21}, the latter question is solved in the positive for Banach spaces with a finite dimensional decomposition but the general problem is still open. The main result of this work is a positive solution to a Hölder version of the problem. Indeed, an $\alpha$-Hölder retraction onto such a subset is constructed for every positive $\alpha<1$. In \cite{HM22}, the bounded approximation property is proved to be necessary for a separable Banach space to have a Lipschitz retraction onto a generating compact convex subset satisfying  an additional mild assumption on its shape. This hints to the possibility of a negative solution for the case $\alpha=1$ (Lipschitz case).

It is worth mentioning that in \cite{CCW21} and later but independently in \cite{HM21}, compact convex subsets of Banach spaces are proven to be absolute uniform retracts. In fact, for separable Banach spaces it is enough to consider the nearest point map respect to a URED renorming of the space. This means that the modulus of continuity of the retraction depends on the modulus of convexity of the given URED norm and it is independent of the shape of the retract. Here, a completely different approach is used, allowing us to tune up the shape of the retract and obtain a much better estimate on the modulus of continuity of the retraction.

In section \ref{Holder} we analyse the compact convex subsets of Banach spaces in terms of flatness. That is, we are interested in the heights of a compact set respect to its finite dimensional sections (see Definition \ref{flat} below). In particular, for every $\alpha\in(0,1)$ we find a flatness condition such that every compact convex set under this condition is an absolute $\alpha$-Hölder retract (Theorem \ref{mainth2}). This leads to the main result of the paper Theorem \ref{mainth}, namely,  every separable Banach space $X$ has an $\alpha$-Hölder generating compact convex retract. Moreover, we prove that the provided generating compact convex subset of $X$ is not a Lipschitz retract whenever $X$ does not enjoy the bounded approximation property.


The second main result of this note is devoted to nearest point maps (also called as metric projections or proximity mappings). It was initially observed by Phelps in \cite{Phe58} that the nearest point map from a Hilbert space $X$ onto any closed convex subset is $1$-Lipschitz (nonexpansive). In fact, the latter property characterizes Hilbert spaces $X$ whenever $\text{dim}(X)\ge3$. Later, Björnestal generalized this result to Banach spaces with a uniformly rotund norm. He showed in \cite{Bjo79} that the nearest point map from a bounded subset of a Banach space onto a closed convex subset is uniformly continuous whenever the norm in $X$ is uniformly rotund. Brown \cite{Bro74} and Veselý \cite{Ves91} proved that the preceding is not true for strictly convex spaces. In particular, they found examples of strictly convex reflexive spaces and closed convex (noncompact) subsets of these spaces for which the nearest point map is not continuous.

However, for compact convex subsets the nearest point map from a strictly convex space is always continuous. Moreover, it is proved in \cite{HM21} that the nearest point map from any bounded subset of $X$ onto a compact convex subset $K$ of $X$ is uniformly continuous provided the norm in $X$ is uniformly rotund in the direction $x$ for every $x$ in the closed linear span of $K$. We prove that, as expected, the latter result cannot be generalized to rotund norms.

More precisely, in Section \ref{NPM} we construct a strictly convex Banach space $X$ arbitrarily close to $\ell_2$ (for the Banach Mazur distance) and a finite dimensional compact subset of $X$ for which the nearest point map is not uniformly continuous even when restricted to bounded sets (Theorem \ref{mainNPM}). This highly contrasts with the fact that every nearest point map for the usual norm of $\ell_2$ onto a closed convex subset  is nonexpansive.

\subsection{Definitions and notation.}

Let $(M,d_M)$ and $(N,d_N)$ be two metric spaces and let $f:M\to N$ be an arbitrary mapping. We consider the modulus of continuity of $f$ as the function $\omega_f:[0,\infty)\to[0,\infty]$ given by
$$\omega_f(t)=\sup\big\{d_N\big(f(x),f(y)\big)\;:\;d_M(x,y)\le t\big\},$$
where the supremum is considered to be infinite whenever it does not exist. We will say that $f$ is uniformly continuous if $\omega_f$ is continuous at $t=0$. Given some $\alpha\in(0,1]$, we will say that $f$ is $\alpha$-Hölder whenever there exists $C>0$ such that $\omega_f(t)\le Ct^\alpha$ for every $t$. In the particular case of $\alpha=1$ we say that $f$ is Lipschitz.

A retraction from a metric space $(M,d)$ onto a subset $N\subset M$ is a map $R:M\to N$ satisfying that $R(x)=x$ for every $x\in N$. The image of a retraction is called a retract. If the retraction is uniformly continuous (resp. $\alpha$-Hölder or Lipschitz) then we say that its image is a uniform (resp. $\alpha$-Hölder or Lipschitz) retract. We say that $N$ is an absolute uniform (resp. $\alpha$-Hölder or Lipschitz) retract if it is a uniform (resp. $\alpha$-Hölder or Lipschitz) retract of every metric space containing it.  In this setting, a nearest point map is any retraction $R:M\to N$ such that for every $x\in M$,
$$d(R(x),x)=\inf\limits_{y\in N} d(y,x).$$
A nearest point map may not be unique or may not exist in some cases.

We will refer as a net of a metric space $(M,d)$ to any subset $N\subset M$ such that there is $a,b\in\R^+$ satisfying the next two properties:
\begin{itemize}
\item $N$ is $a$-separated, that is, $d(x,y)\ge a$ for every $x\neq y\in N$.
\item $N$ is $b$-dense, that is, for every $x\in M$ there is $y\in N$ such that $d(x,y)\le b$.
\end{itemize}
In this case we say that $N$ is an $(a,b)$-net of $M$. If $M$ is locally compact then for every $\ep>0$ there is an $(\ep,\ep)$-net of $M$. In fact, it is enough to take a maximal $\ep$-separated subset of $M$.

If $X$ is a Banach space, we say that a subset $S$ of $X$ is a generating subset whenever the closed linear span of $S$ is $X$. We denote the closed linear span as $[S]$. A fundamental sequence of $X$ is any generating countable subset of $X$.

A Banach space $X$ is said to be strictly convex whenever the closed unit ball of $X$ is strictly convex. In this case, we say that the norm $||\cdot||$ of $X$ is rotund, which is equivalent to satisfying the following inequality
$$||x+y||<||x||+||y||\;\;\;\forall x,y\in S_X\;\text{ with }\;x\neq y.$$

For the remaining general concepts and results of Banach space theory we refer to \cite{Fab1}. Throughout the entire note we will follow the terminology and notation used in \cite{Fab1}. For the background on Lipschitz and uniformly continuous retractions we refer the reader to the first two chapters of the authoritative monograph \cite{BL2000}.

\section{Hölder retractions onto compact sets}\label{Holder}

In this section we will prove the following result.

\begin{theorem}\label{mainth}
For every separable Banach space $X$ and every $\alpha\in(0,1)$ there is a generating compact convex subset $K$ of $X$ which is an $\alpha$-Hölder retract of $X$.
\end{theorem}

In order to keep a clean notation, if $(x_n)_{n\in\N}$ is any sequence of a Banach space, we agree that $\sum\limits_{i=n}^mx_i=0$  and $[x_i]_{i=n}^m=0$ whenever $n>m$.

\subsection{The compact convex set.}

Let us first focus on the compact set that will be the target of our retraction. The following concept was previously considered in \cite{HM21} and \cite{HM22} and has proved to be very useful to measure the `asymptotic size' of a compact subset of a Banach space.

\begin{definition}\label{flat}
Given a compact subset $K$ of a Banach space $X$ with $0\in K$ and a fundamental sequence $\beta=(e_n)$ of $[K]$, we define the sequence of heights $(h^\beta_n)_{n\in\N\cup\{0\}}$ of $K$ relative to $\beta$ as
$$h_n^\beta=\sup\{d(x,K\cap[e_i]_{i=1}^n)\;:\;x\in K\}.$$
\end{definition}

Let us mention now that since $K$ is compact it holds that $\lim_nh^\beta_n=0$ for every choice of $\beta$. We will say that $K$ is $(r_n)$-flat for some null sequence of positive real numbers $(r_n)$ whenever there exists a fundamental sequence $\beta$ of $[K]$ such that $h_n^\beta\le r_n$ for every $n\in\N\cup\{0\}$. We will say that $K$ is flat if it is $(3+\ep)^{-n}$-flat for some $\ep>0$. See \cite{HM22} for some results on flat sets, although our definition here is slightly stronger than the one given in \cite{HM22}.

\begin{example}\label{mainex}
For every decreasing null sequence of positive real numbers $(r_n)$ and every separable Banach space $X$, there is an $(r_n)$-flat generating compact convex subset of $X$. In fact, consider $(e_n)\subset S_X$ a fundamental sequence of $X$ and let $E_n=[e_i]_{i=1}^n$ for every $n\in\N\cup\{0\}$. Then, we claim that both
$$K_1=\overline{\bigcup\limits_{n\in\N}\sum\limits_{k=1}^n2^{-k}[-r_{k-1},r_{k-1}]e_k}\;\;\;\;,\;\;\;\;K_2=\cco\bigg(\bigcup\limits_{n\in \N}[-r_{n-1},r_{n-1}]e_n\bigg)$$
are $(r_n)$-flat respect to the fundamental sequence $\beta=(e_n)$.
\begin{proof}
We prove the result for $K_1$, the proof for $K_2$ is essentially the same. Let us take $x=\sum\limits_{k=1}^m\lambda_{k-1}e_k$ for some $m\in\N$ where $|\lambda_{k-1}|\le2^{-k}r_{k-1}$ for $k=1,\dots,m$. Let $n\in\N\cup\{0\}$ be arbitrary. If $n\ge m$ then $x\in K_1\cap E_n$ and so $d(x,K_1\cap E_n)=0<r_n$. Otherwise,
$$\begin{aligned}d(x,K_1\cap E_n)\le&\Big|\Big|x-\sum\limits_{k=1}^n\lambda_{k-1}e_k\Big|\Big|=\Big|\Big|\sum\limits_{k=n+1}^m\lambda_{k-1}e_k\Big|\Big|\le\sum\limits_{k=n}^{m-1}|\lambda_k|\\\le& \sum\limits_{k=n}^{m-1}2^{-k-1}r_k\le r_n\sum\limits_{k=n}^{m-1}2^{-k-1}\le r_n.\end{aligned}$$
Hence, $d(x,K_1\cap E_n)\le r_n$ for every $n\in\N\cup \{0\}$ and $x\in K_1$. Therefore,
$$h^\beta_n=\sup\{d(x,K_1\cap[e_i]_{i=1}^n)\;:\;x\in K\}\le r_n\;\;\;\forall n\in\N\cup\{0\}$$
and we are done.
\end{proof}
\end{example}

From now on during this section $K$ is an arbitrary but fixed $(r_n)$-flat convex subset of a Banach space $X$ where $(r_n)\subset \R^+$ is an arbitrary decreasing null sequence. Let us consider $\beta=(e_n)$ the fundamental sequence of $[K]$ satisfying that $h_n^\beta\le r_n$ for every $n\in\N$. We will denote $E_n=[e_i]_{i=1}^n$ for $n\in\N$. Given $\ep>0$ we define here
$$n(\ep)=\min\{n\in\N\cup\{0\}\;:\; r_n\le\ep \}.$$
Now, for every $\ep>0$ we consider an $(\ep,\ep)$-net $N_\ep=(x_i^\ep)_{i=1}^{m_\ep}$ of $K\cap E_{n(\ep)}$. Clearly, $N_\ep$ is a $(\ep,2\ep)$-net of $K$ for every $\ep>0$.

In the next subsections we are going to construct a retraction from $X$ onto $K$.

\subsection{Whitney-type partitions of unity.}

In this subsection we are going to follow the construction of a partition of unity in $X\setminus K$ given in \cite{BDS21}, which follow the path traced by Whitney in \cite{Whi34} and was intended to provide a simpler proof of a celebrated result of Lee and Naor \cite{LN05}. Let us recall that their construction is done for doubling metric spaces which may not be the case of $K$. This partition of unity will be crucial later for the definition of the retraction.

We set $\ep_n=2^{-n}$ for every $n\in \Z$ and rename $N_{\ep_n}$ as $N_n$ relabeling its elements as $x_i^n:=x_i^{\ep_n}$ for every $i\in\{1,\dots,m_{\ep_n}=m_n\}$. Then, let us consider for each $n\in\Z$ and $i\le m_n$ the sets
$$\widetilde{V}_i^n=\{x\in X\;:\;\ep_n\le d(x,K)<\ep_{n-1}\;,\;d(x,x_i^n)=\min\limits_{j}d(x,x_j^n)\}$$
and
$$V_i^n=\{x\in X\;:\;d(x,\widetilde{V}_i^n)\le\ep_{n+1}\}.$$
\begin{lemma}\label{basics}
The family $\mathcal{F}=\{V_i^n\;:\;n\in\Z,\;i\in\{1,\dots,m_n\}\}$ is a locally finite cover of $X\setminus K$. Moreover, if $x\in V_i^n\in\mathcal{F}$ then,
\begin{enumerate}
\item $d(x,K)/5\le||x-x_i^n||\le 9d(x,K)$.\label{part1}
\item $\#\{V\in\mathcal{F}\;:\;x\in V\}\le 5\cdot 20^{n(d(x,K)/10)}$.\label{part2}
\item $d(x,K)/4\le\max_{V\in F}d(x,V^c)\le d(x,K)$.\label{part3}
\end{enumerate}
\begin{proof}
Let us take $x\in V_i^n$.
The computations proving \eqref{part1} and \eqref{part3} are already done in \cite{BDS21} but we include them here for the sake of completeness.
We start proving \eqref{part1}. An easy application of the triangular inequality yields
\begin{equation}\label{eq0part}d(x,K)\le d(x,\widetilde{V}_i^n)+\sup \limits_{y\in\widetilde{V}_i^n}d(y,K)\le\ep_{n+1}+\ep_{n-1}=5\ep_{n+1}.\end{equation}
Now, take $\delta>0$, $y\in \widetilde{V}_i^n$ and $k\in K$ such that $||x-y||\le \ep_{n+1}+\delta$ and $d(y,K)\ge ||y-k||-\delta$. Consider also $j\le m_n$ such that $||k-x_j^n||\le 2\ep_n=\ep_{n-1}$. Then,
\begin{equation}\label{eq1part}||x-x_i^n||\ge d(x,K)\ge d(y,K)-||x-y||\ge \ep_{n}-\ep_{n+1}-\delta=\ep_{n+1}-\delta\end{equation}
and
$$||y-x_i^n||\le||y-x_j^n||\le ||y-k||+||k-x_j^n||\le d(y,K)+\delta+\ep_{n-1}\le2\ep_{n-1}+\delta.$$
Therefore,
\begin{equation}\label{auxeq}||x-x_i^n||\le ||x-y||+||y-x_i^n||\le\ep_{n+1}+2\ep_{n-1}+2\delta=9\ep_{n+1}+2\delta.\end{equation}
By \eqref{eq1part}, $\ep_{n+1}\le d(x,K)+\delta$ so the previous inequality yields
\begin{equation}\label{eq2part}||x-x_i^n||\le9d(x,K)+11\delta.\end{equation}
Since $\delta>0$ was arbitrary \eqref{part1} follows from \eqref{eq0part}, \eqref{eq1part} and \eqref{eq2part}. In particular from \eqref{auxeq} the sharper estimate $||x-x_i^n||\le 9\ep_{n+1}$ is deduced. We will use this last estimate to prove \eqref{part2}. Let us call $A_k=\{j\in\N\;:\;j\le m_k\;,\;x\in V_j^k\}$ for every $k\in\Z$. We will first obtain an upper bound for $\#A_k$. Let us fix a pair $j,k$ such that $x\in V_{j}^{k}$. If $x\in V_{l}^k$ for some other $l\le m_k$ then
$$||x_{j}^k-x_{l}^k||\le ||x_{j}^k-x||+||x_{l}^k-x||\le 18\ep_{k+1}=9\ep_{k}.$$
In particular, if for each $y\in E_{n(\ep_k)}$ and $r\in\R^+$ we denote the ball in $E_{n(\ep_k)}$ of radius $r$ centered at $y$ as $B(y,r)$, then
$$B(x_{l}^k,\ep_k/2)\subset B(x_{j}^k,10\ep_k),$$
for every $l\le m_k$ such that $x\in V_{l}^k$. In other words,
$$A_k\subset\{l\in\N\;:\;l\le m_k,\;B(x_{l}^k,\ep_k/2)\subset B(x_{j}^k,10\ep_k)\}=:B_k.$$
Now, if $\lambda$ stands for the Lebesgue measure in $E_{n(\ep_k)}$ then by the change of variables theorem,
\begin{equation}\label{meas1}\lambda\big( B(x_{j_1}^k,10\ep_k) \big)=20^{n(\ep_k)}\lambda\big( B(x_{j_2}^k,\ep_k/2) \big)\;\;\;\forall j_1,j_2\in B_k.\end{equation}
Also, since $N_k$ is $\ep_k$-separated we have that
\begin{equation}\label{meas2}\lambda\big(B(x_{j_1}^k,\ep_k/2)\cap B(x_{j_2}^k,\ep_k/2)\big)=0\;\;\;\;\forall j_1,j_2\in B_k.\end{equation}
From \eqref{meas1} and \eqref{meas2} we deduce that
$$\begin{aligned}\lambda\big(B(x_j^k,10\ep_k)\big)\ge&\lambda\Big(\bigcup\limits_{l\in B_k}B(x_l^k,\ep_k/2)\Big)\\=&\sum\limits_{l\in B_k}\lambda\big(B(x_l^k,\ep_k/2)\big)=\#B_k\frac{\lambda\big(B(x_j^k,10\ep_k)\big)}{20^{n(\ep_k)}}.\end{aligned}$$
Multiplying by $\frac{20^{n(\ep_k)}}{\lambda\big(B(x_j^k,10\ep_k)\big)}$ on both sides of the previous inequality we get $\#B_k\le20^{n(\ep_k)}$ so that we ultimately obtain
$$\#A_k\le20^{n(\ep_k)}.$$
Now, if $j,k$ are such that $x\in V_j^k$ then clearly $|k-n|\le2$. Hence,
$$\begin{aligned}\#\{V\in\mathcal{F}\;:\;x\in V\}\le&\sum\limits_{\substack{k\in\N\\|k-n|\le2}}\#A_{k}\le 5\cdot20^{n(\ep_{n+2})}.\end{aligned}$$
The inequality \eqref{part2} follows now from the fact that $n(\ep)$ is a decreasing function and that by \eqref{eq0part}, $\ep_{n+2}=\ep_{n+1}/2\ge d(x,K)/10$. Finally, we pass to the proof of \eqref{part3}. The inequality $\max_{V\in \mathcal{F}}d(x,V^c)\le d(x,K)$ is true since $K\subset V^c$ for all $V\in\mathcal{F}$. To prove the other inequality, let us take $j,k$ such that $x\in \widetilde{V}_j^k$. Then $d(x,(V_j^k)^c)\ge \ep_{k+1}$ and thus
$$d(x,K)/4<\ep_{k-1}/4=\ep_{k+1}\le d(x,(V_j^k)^c)\le\max\limits_{V\in\mathcal{F}}d(x,V^c).$$
\end{proof}
\end{lemma}

Once we have defined the locally finite cover of $X\setminus K$, it is time to define the associated functions that will form our partition of unity. As an additional property, we are going to prove that these functions are $L(\ep)$-Lipschitz when restricted to $\{x\in X\;:\;d(x,K)\ge\ep\}$ for every $\ep>0$. The difference between our case and the case of doubling metric spaces arises when considering  small values of $\ep$ since we will see that $\sup_{\ep>0}L(\ep)=\infty$.

The partition of unity is now defined in a very natural way. For every $n\in\Z$, $i\in\{1,\dots, m_n\}$ and $x\in X\setminus K$, we set
$$\varphi_i^n(x)=\frac{d(x,(V_i^n)^c)}{\sum\limits_{k,j}d(x,(V_j^k)^c)}.$$
It is clear that the pair $\big(\mathcal{F},(\varphi_i^n)_{i,n}\big)$ is a partition of unity. Let us quantify its Lipschitz behaviour.

\begin{lemma}\label{partition}
For every $x,y\in X\setminus K$,
$$\sum\limits_{i,n}|\varphi_i^n(x)-\varphi_i^n(y)|\le \frac{40\big( 20^{n(d(x,k)/10)}+ 20^{n(d(y,k)/10)}\big)}{\max\{d(x,K),d(y,K)\}}||x-y||.$$
\begin{proof}
For aesthetic purposes, we rename and relabel the family of functions $\{d(x,(V_i^n)^c)\}_{\substack{n\in\Z\\i\in\{1,\dots,m_n\}}}$ as $\{d_i(x)\}_{i\in\N}$. Clearly, $|d_i(x)-d_i(y)|\le||x-y||$ for every $i\in\N$. Hence, we compute
$$\begin{aligned}\bigg| \frac{d_i(x)}{\sum\limits_jd_j(x)}- \frac{d_i(y)}{\sum\limits_jd_j(y)} \bigg|\le&\bigg| \frac{d_i(x)}{\sum\limits_jd_j(x)}-\frac{d_i(y)}{\sum\limits_jd_j(x)}\bigg|+\bigg| \frac{d_i(y)}{\sum\limits_jd_j(x)}-\frac{d_i(y)}{\sum\limits_jd_j(y)}\bigg|\\=&  \frac{|d_i(x)-d_i(y)|}{\sum\limits_jd_j(x)}+d_i(y) \frac{\Big|\sum\limits_jd_j(x)-d_j(y)\Big|}{\Big(\sum\limits_jd_j(x)\Big)\Big(\sum\limits_jd_j(y)\Big)}\\\le& \Bigg(  \frac{1}{\sum\limits_jd_j(x)}+\frac{d_i(y)\#\{V\in\mathcal{F}\;:\;\{x,y\}\cap V\neq\emptyset\}}{\Big(\sum\limits_jd_j(x)\Big)\Big(\sum\limits_jd_j(y)\Big)}  \Bigg)||x-y||.\end{aligned}$$
Taking into account the previous inequality and the fact that $\varphi_i^n(x)-\varphi_i^n(y)=0$ if $\{x,y\}\cap V_i^n=\emptyset$ we deduce that
$$\sum\limits_{i,n}|\varphi_i^n(x)-\varphi_i^n(y)|\le \frac{2\#\{V\in\mathcal{F}\;:\;\{x,y\}\cap V\neq\emptyset\}}{\sum\limits_jd_j(x)}||x-y||.$$
Clearly, by \eqref{part2} of Lemma \ref{basics},
$$\begin{aligned}\#\{V\in\mathcal{F}\;:\;\{x,y\}\cap V\neq\emptyset\}\le&\#\{V\in\mathcal{F}\;:\;x\in V\}+\#\{V\in\mathcal{F}\;:\;y\in V\}\\\le&5\cdot\big(20^{n(d(x,K)/10)}+20^{n(d(y,K)/10)}\big),\end{aligned}$$
and by \eqref{part3} of Lemma \ref{basics},
$$\sum_jd_j(x)\ge\max_{V\in\mathcal{F}}d(x,V^c)\ge d(x,K)/4.$$
Hence,
$$\sum\limits_{i,n}|\varphi_i^n(x)-\varphi_i^n(y)|\le\frac{40\big( 20^{n(d(x,k)/10)}+ 20^{n(d(y,k)/10)}\big)}{d(x,K)}||x-y||.$$
Interchanging the roles of $x$ and $y$ one may have
$$\sum\limits_{i,n}|\varphi_i^n(x)-\varphi_i^n(y)|\le\frac{40\big( 20^{n(d(x,k)/10)}+ 20^{n(d(y,k)/10)}\big)}{d(y,K)}||x-y||,$$
which finishes the proof.
\end{proof}
\end{lemma}

\subsection{The retraction.}

Now, we are ready to define the retraction onto $K$. Let us consider the map $R:X\to K$ given by
$$R(x)=\begin{cases}\sum\limits_{i,n}\varphi_i^n(x)x_i^n\;\;&\text{ if }x\in X\setminus K,\\ x&\text{ if }x\in K.\end{cases}$$
Clearly, if $x\in X\setminus K$ then by \eqref{part1} of Lemma \ref{basics},
$$||R(x)-x||\le\sum\limits_{i,n}\varphi_i^n(x)||x_i^n-x||\le 9d(x,K).$$
Hence, we have the following Fact \ref{npm}.
\begin{fact}\label{npm}
For every $x\in X$,
$$||R(x)-x||\le 9d(x,K).$$
\end{fact}

\begin{prop}\label{proplip}
For every $x,y\in X\setminus K$,
$$||R(x)-R(y)||\le 760\big(20^{n(d(x,K)/10)}+20^{n(d(y,K)/10)}\big)||x-y||.$$
\begin{proof}
Let us distinguish two different cases. The first case is when $||x-y||\ge\max\{d(x,K),d(y,K)\}$. In this case, from Fact \ref{npm} we deduce that
$$\begin{aligned}||R(x)-x||\le9d(x,K)\le9\max\{d(x,K),d(y,K)\}\le9||x-y||.\end{aligned}$$
Similarly, $||R(y)-y||\le9||x-y||$ so that
$$||R(x)-R(y)||\le||R(x)-x||+||R(y)-y||+||x-y||\le19||x-y||.$$
Now, we may assume that $||x-y||\le\max\{d(x,K),d(y,K)\}$. In this case we fix $i_0,n_0$ such that $x\in V_{i_0}^{n_0}$. If $x\in V_i^n$ for some other pair $i,n$ then by \eqref{part1} in Lemma \ref{basics} we obtain
$$||x_i^n-x_{i_0}^{n_0}||\le||x_i^n-x||+||x_{i_0}^{n_0}-x||\le18\max\{d(x,K),d(y,K)\}.$$
On the other hand, if the pair $i,n$ is such that $y\in V_i^n$ then
$$||x_i^n-x_{i_0}^{n_0}||\le ||x_i^n-y||+||x_{i_0}^{n_0}-x||+||x-y||\le 19\max\{d(x,K),d(y,K)\}.$$
Hence, by Lemma \ref{partition},
$$\begin{aligned}||R(x)-R(y)||=&\Big|\Big| \sum\limits_{i,n}\big(\varphi_i^n(x)-\varphi_i^n(y)\big)(x_i^n-x_{i_0}^{n_0}) \Big|\Big|\\\le& \sum\limits_{i,n}|\varphi_i^n(x)-\varphi_i^n(y)|\cdot19\max\{d(x,K),d(y,K)\}\\\le& 760\big( 20^{n(d(x,k)/10)}+ 20^{n(d(y,k)/10)}\big)||x-y||.\end{aligned}$$
\end{proof}
\end{prop}

This result is showing how the retraction losses its Lipschitz constant when the points get close to $K$. In fact, the slower $n(d(x,K)/10)$ grows to infinity when $d(x,K)$ tends to zero, the slower the Lipschitz constant is lost when the points are approaching $K$. The good news are that the sequence $(r_n)$ controls how fast $n(d(x,K)/10)$ grows. This means that the sequence $(r_n)$ is monitoring how fast the Lipschitz constant is lost.

In order to present the latter in a precise way, we compute an upper bound for the modulus of continuity of $R$.

\begin{theorem}\label{mainan}
For every $t\in\R^+$,
$$\omega_R(t)\le 1520\cdot20^{n(t/20)}t.$$
\begin{proof}
Let us take $y\in X$. We are going to distinguish between several different cases.

\textit{Case 1}. The case when $x\in K$. 
Clearly $||x-y||\ge d(y,K)$. Therefore, by Fact \eqref{npm},
$$\begin{aligned}||R(x)-R(y)||\le& ||R(x)-x||+||R(y)-y||+||x-y||\\\le&9d(y,K)+||x-y||\le 10||x-y||.\end{aligned}$$

\textit{Case 2}. The case when $x\in X\setminus K$ and $d(y,K)<d(x,K)/2$.
By the triangle inequality $||x-y||\ge d(x,K)-d(y,K)\ge d(x,K)/2$. Hence, using again Fact \eqref{npm},
$$\begin{aligned}||R(x)-R(y)||\le& ||R(x)-x||+||R(y)-y||+||x-y||\\\le&9(d(x,K)+d(y,K))+||x-y||\le 28||x-y||.\end{aligned}$$

\textit{Case 3}. The case when $x\in X\setminus K$, $d(y,K)\ge d(x,K)/2$ and $||x-y||\ge d(x,K)$.
Here, $d(y,K)\le d(x,K)+||x-y||\le 2||x-y||$. Hence,
$$\begin{aligned}||R(x)-R(y)||\le& ||R(x)-x||+||R(y)-y||+||x-y||\\\le&9(d(x,K)+d(y,K))+||x-y||\le 55||x-y||.\end{aligned}$$

\textit{Case 4}. The case when $x\in X\setminus K$, $d(y,K)\ge d(x,K)/2$ and $||x-y||< d(x,K)$.
By Proposition \ref{proplip},
$$\begin{aligned}||R(x)-R(y)||\le&760\big(20^{n(d(x,K)/10)}+20^{n(d(y,K)/10)}\big)||x-y||\\\le& 1520\cdot 20^{n(d(x,K)/20)}||x-y||\le1520\cdot20^{n(||x-y||/20)}||x-y||.\end{aligned}$$
\end{proof}
\end{theorem}

\begin{remark}
We did not strive during this section to obtain the best estimate for $\omega_R(t)$ in Theorem \ref{mainan}. Very probably sharper estimates can be derived.
\end{remark}

Since all the previous constructions and results are proven for an arbitrary $(r_n)$-flat convex subset of an arbitrary Banach space, we may state and prove the following theorem.

\begin{theorem}\label{mainth2}
If $\alpha\in(0,1)$ then every $\big(20^{\frac{n}{\alpha-1}}\big)$-flat convex subset of a Banach space $X$ is an absolute $\alpha$-Hölder retract.
\begin{proof}
Since every metric space is isometrically embeddable into a Banach space, it is enough to find an $\alpha$-Hölder retraction from $X$ onto $K$. By Theorem \ref{mainan}, it is enough to show that  for $r_n=20^{\frac{n}{\alpha-1}}$ there is a constant $C>0$ such that for every $t>0$,
$$n(t/20)\le \log_{20}(t^{\alpha-1})+C.$$
It is immediate that in this case,
$$n(t)\le \min\{n\in\N\;:\;n\ge\log_{20}(t^{\alpha-1})\}\le\log_{20}(t^{\alpha-1})+1.$$
Therefore,
$$n(t/20)\le \log_{20}\big((t/20)^{\alpha-1}\big)+1=\log_{20}(t^{\alpha-1})+2-\alpha.$$
\end{proof}
\end{theorem}

\begin{remark}
A result like Theorem \ref{mainth2} cannot be expected for every compact convex set. In fact, if a compact convex set $K$ is an absolute $\alpha$-Hölder retract for every $\alpha<1$ then it is clearly an absolute Lipschitz retract. Hence, an extra assumption on the shape of $K$ is needed.
\end{remark}

Theorem \ref{mainth} follows now as a consequence of Theorem \ref{mainth2} and Example \ref{mainex}. We state it here in the form of a corollary.

\begin{corollary}\label{mainco}
For every separable Banach space $X$ and every $\alpha\in(0,1)$ there is a generating compact convex set $K$ of $X$ which is an $\alpha$-Hölder retract of $X$. Moreover, $K$ can be chosen to be flat.
\end{corollary}

This result is to some extent optimal and underlines the differences between Lipschitz and Hölder retractions. More precisely,

\begin{corollary}
For every separable Banach space $X$ without the bounded approximation property and for every $\alpha\in(0,1)$ there is a generating compact convex subset of $X$ which is an $\alpha$-Hölder retract of $X$ but not a Lipschitz retract of $X$.
\begin{proof}
From \cite{HM22} we know that a separable Banach space admitting a Lipschitz retraction onto a flat generating convex subset has the bounded approximation property. Hence, if $X$ does not enjoy the bounded approximation property then it is enough to consider the flat generating compact convex subset given by Corollary \ref{mainco}.
\end{proof}
\end{corollary}

Let us remind that by a famous example of Enflo \cite{Enf73} we know of the existence of separable Banach spaces without the bounded approximation property.

The careful reader may have noticed that the definition of the retraction in this section is based on nearest point maps. In fact, the retraction is defined as an average of nearest point maps onto some carefully chosen nets, where each set of our locally finite covering of $X\setminus K$ had as a target its nearest point from a net of a finite dimensional section of $K$. In the following section we will see that the mentioned averaging process seems to be crucial.

\section{Nearest point map}\label{NPM}

In \cite{HM21},  it is proved that the nearest point map from any bounded subset of $X$ onto a compact convex subset $K$ of $X$ is uniformly continuous provided the norm in $X$ is uniformly rotund in the direction $x$ for every $x\in S_X\cap[K]$. The natural conjecture is that this result can not be pushed for rotund norms, not even when the targeted subset is a finite dimensional compact convex set. The goal of this section is to provide an example of a strictly convex Banach space $X$ and a finite dimensional compact convex subset $K$ of $X$ such that the nearest point map from $X$ onto $K$ is not uniformly continuous even when restricted to bounded sets. More precisely, we are going to prove the next stronger result.

\begin{theorem}\label{mainNPM}
For every $\delta>0$ and $K$ nontrivial finite dimensional compact convex susbet of $\ell_2$ there is a rotund norm $|\cdot|$ in $\ell_2$ and an element $p\in K$ such that
$$d_{BM}\big((\ell_2,|\cdot|),(\ell_2,||\cdot||_2)\big)<\delta,$$
and the nearest point map from $\ell_2$ onto $K$ given by the norm $||\cdot||$ is not uniformly continuous even when restricted to $p+\ep B_{\ell_2}$ for any $\ep>0$.
\end{theorem}

Here $d_{BM}$ stands for the Banach mazur distance, that is, for every pair of isomorphic Banach spaces $X$, $Y$ we have,
$$d_{BM}(X,Y)=\inf\big\{\ln\big(||T||\cdot||T^{-1}||\big)\;:\;T\in\mathcal{L}(X,Y)\text{ is an isomorphism}\big\}.$$

It is worth mentioning that the nearest point map from a strictly convex Banach space onto a closed convex subset is unique. Also, it is straightforward to check that the nearest point map to a compact set is continuous whenever it is unique. Hence, the nearest point map mentioned in Theorem \ref{mainNPM} is continuous but not uniformly continuous.

During this section, given a convex subset $C$ of a Banach space $X$, we will denote $\text{Int}(C)$ for the topological interior with the topology of $[C]$ inherited by the norm topology of $X$.

\subsection{The construction.}

Let us denote $(e_n)$ the canonical basis of $\ell_2$. We first state and prove the following lemma concerning the geometry of $B_{\ell_2}$. From now on, $||\cdot||_2$ denotes the $\ell_2$ norm and $\langle\cdot,\cdot\rangle$ stands for the canonical inner product in $\ell_2$.

\begin{lemma}\label{geo2}
Let $P\subset \ell_2\setminus B_{\ell_2}$. Then,
$$\sup\limits_{x\in \cco(P\cup B_{\ell_2})\setminus B_{\ell_2}}d_2\big(x,\co(P)\big)\le \sqrt{\sup||P||_2^2-1},$$
where $d_2$ is the distance for the $\ell_2$ norm.
\begin{proof}
Let us take  $x\in \co\big(P\cup B_{\ell_2}\big)\setminus B_{\ell_2}$. Then, by definiton, there is $y\in \co(P)$, $z\in S_{\ell_2}$ and $\lambda\in (0,1]$ such that
$$x=(1-\lambda)z+\lambda y.$$
Moreover, since $x\notin B_{\ell_2}$ we may assume that $z$ is chosen so that the segment connecting $x$ and $z$ does not intersect $\text{Int}\big(B_{\ell_2}\big)$. Hence,  $(1-t)z+ty\notin \text{Int}\big(B_{\ell_2}\big)$ for every $t\in(0,\lambda]$. Then, for every $t\in(0,\lambda]$ we have that
$$1\le||(1-t)z+ty||_2^2=(1-t)^2+t^2||y||_2^2+2t(1-t)\langle y,z \rangle.$$
Equivalently, we get that for each $t\in(0,\lambda]$,
\begin{equation}\label{limitt}\langle y,z \rangle\ge\frac{2t-t^2(||y||_2^2+1)}{(2t-2t^2)}.\end{equation}
Now we take limit when $t\to0^+$ in \eqref{limitt} and get that $\langle y,z\rangle\ge1.$ Hence,
$$d_2(x,\co(P))^2\le||y-x||_2^2\le||y-z||_2^2=||y||_2^2+1-2\langle y,z\rangle\le||y||_2^2-1\le\sup||P||_2^2-1.$$
Since $x\in \co\big(P\cup B_{\ell_2}\big)\setminus B_{\ell_2}$ was arbitrary we get the desired inequality.
\end{proof}
\end{lemma}

We will use the following concept in the setting of $\ell_2$.

\begin{definition}
The modulus of convexity of a Banach space $(X,||\cdot||)$ is a function $\delta:[0,2]\to[0,1]$ given by
$$\delta(\ep)=\inf\Big\{1-\Big|\Big| \frac{x+y}{2} \Big|\Big|\;:\;x,y\in S_X,\;||x-y||\ge\ep\Big\},$$
where $S_X$ denotes the unit sphere of $(X,||\cdot||)$.
\end{definition}

Let us remind the following well known fact that will also be used later.

\begin{fact}\label{modcon}
If $\delta_2:[0,2]\to[0,1]$ denotes the modulus of convexity of $(\ell_2,||\cdot||_2)$ then
$$\delta_2(\ep)=1-\sqrt{1-\Big(\frac{\ep}{2}\Big)^2}\;\;\;\;\;\forall\ep\in[0,2].$$
\end{fact}


For later purposes we are going to need a positive `deviation parameter' $\delta\le1/48$. Then, it is straightforward to see that $0<(1+2\delta)^2-(1-10\delta)^2<1$. Hence,
\begin{equation}\label{deltaineq}(1+2\delta)\sqrt{1-\Big(\frac{1-10\delta}{1+2\delta}\Big)^2}<1.\end{equation}
Let us start with the construction of the rotund norm. We define for every $n\in\N$ and $i=1,2$ the elements
$$x_{i,n}=e_{2n}+(-1)^{i+1}\frac{\delta}{n}e_{2n+1}\;\;\;,\;\;\;z_{i,n}=x_{i,n}+(-1)^i\delta e_{1}.$$
Now, for every $n\in\N$ we consider the sets 
$$P_n=\{\pm z_{1,n},\pm z_{2,n}\}\;\;\;,\;\;\;\widetilde B_n=\cco\big(P_n\cup B_{\ell_2}\big).$$
The set $\widetilde B_n$ clearly produces a norm $||\cdot||_n$ in $\ell_2$ for every $n\in \N$. It is straightforward to check that
\begin{equation}\label{baseequiv}\frac{1}{1+2\delta}||x||_2\le||x||_n\le||x||_2\;\;\;\forall x\in \ell_2,\;\;\;\forall n\in\N.\end{equation}

\begin{lemma}\label{closed}
Let $n,m\in \N$  be such that $n\neq m$. If $x\in \widetilde B_n\setminus B_{\ell_2}$ and $y\in \widetilde B_m\setminus B_{\ell_2}$ then $||x-y||_2\ge1-10\delta$.
\begin{proof}
Let us define for every $n\in\N$ and $i=1,2$ the sets
$$P^i_n=\big\{(-1)^iz_{1,n},(-1)^iz_{2,n}\big\}\;\;\;,\;\;\;\widetilde B^i_n=\cco\big(B_{\ell_2}\cup P^i_n\big).$$
Then, we claim that for every $n\in\N$,
\begin{equation}\label{counionsign}\widetilde B_n=\widetilde B^1_n\cup\widetilde B^2_n.\end{equation}
Indeed, the containment $\widetilde B^1_n\cup\widetilde B^2_n\subset \widetilde B_n$ is immediate, so we focus on the other containment. If we take $p\in \co\big(P^1_n\cup P^2_n\cup B_{\ell_2}\big)$ then by definition there are $\lambda_1,\lambda_2,\lambda_3\in[0,1]$ with $\sum\limits_{i=1}^3\lambda_i=1$ and there is $(p_1,p_2,p_3)\in P^1_n\times P^2_n\times B_{\ell_2}$ such that $p=\lambda_1p_1+\lambda_2p_2+\lambda_3p_3$. Let us take $i,j\in\{1,2\}$ such that $i\neq j$ and $\lambda_i-\lambda_j\le0$. Clearly $0\le2\lambda_i\le1$ and $0\le\lambda_j-\lambda_i\le1$. Therefore, the righ hand side of the following equality is a convex combination,
$$p=2\lambda_i\frac{p_1+p_2}{2}+(\lambda_j-\lambda_i)p_j+\lambda_3p_3.$$
That is, $p\in \co\big(\big\{p_i,\frac{p_1+p_2}{2},p_3\big\}\big)$.  It is straightforward to see that $\frac{p_1+p_2}{2}\in \frac{P^1_n+P^2_n}{2}\subset B_{\ell_2}$. Hence, $x\in \co\big(\{p_i\}\cup B_{\ell_2}\big)\subset \widetilde B^i_n$, proving \eqref{counionsign}.

Now, let us take $n,m\in\N$ with $n\neq m$ and assume that $x\in \widetilde B_n\setminus B_{\ell_2}$ and $y\in \widetilde B_m\setminus B_{\ell_2}$. By \eqref{counionsign} we know that $x\in \big(\widetilde B^1_n\setminus B_{\ell_2}\big)\cup\big(\widetilde B^2_n\setminus B_{\ell_2}\big)$ and $y\in \big(\widetilde B^1_m\setminus B_{\ell_2}\big)\cup\big(\widetilde B^2_m\setminus B_{\ell_2}\big)$. Let $i,j\in\{1,2\}$ be such that $x\in \widetilde B^i_n\setminus B_{\ell_2}$ and $y\in \widetilde B^j_m\setminus B_{\ell_2}$. Using Lemma \ref{geo2} we obtain that
$$d_2\big(x,P^i_n\big)\le d_2\big(x,\co(P^i_n)\big)+\text{diam}(P^i_n)\le\sqrt{\sup||P^i_n||_2^2-1}+\sqrt{8}\delta\le5\delta.$$
Analogously we get that
$$d_2\big(y,P^j_m\big)\le 5\delta.$$
Hence, by the triangle inequality we conclude that
$$||x-y||_2\ge d_2(P^i_n,P^j_m)-\big( d_2(x,P^i_n)+d_2(y,P^j_m) \big)\ge1-10\delta.$$
\end{proof} 
\end{lemma}

\begin{prop}\label{indexB} If for every $n\in\N$ there is a closed convex symmetric subset $O_n$ of $\ell_2$ such that $B_{\ell_2}\subset O_n\subset \widetilde B_n$ then $\bigcup_{n\in\N}O_n$ is closed, convex and symmetric. Moreover, if $O_n$ is strictly convex then $\bigcup_{n\in\N}O_n$ is also strictly convex.
\begin{proof}
Let us denote $O=\bigcup_n O_n$. The set $O$ is symmetric since $O_n$ is symmetric for every $n\in\N$. To prove that $O$ is closed let us take a sequence $(p_n)\subset O$ such that $p_n\xrightarrow{\ell_2}p\in \ell_2$. If there is a subsequence $p_{\sigma(n)}$ contained in $B_{\ell_2}$ then $p\in B_{\ell_2}\subset O$. Otherwise, we may assume that $p_n\notin B_{\ell_2}$ for every $n\in\N$. We claim that there is $N\in\N$ such that $(p_n)\subset \bigcup\limits_{n=1}^NO_n$. Indeed, if this is not the case then there is a subsequence of $(p_n)$ which we represent by $(p_n)$ and a strictly increasing $\sigma:\N\to\N$ such that $p_n\in O_{\sigma(n)}\setminus B_{\ell_2}$ for every $n\in\N$. Since $O_{\sigma(n)}\subset \widetilde B_{\sigma(n)}$, by Lemma \ref{closed} and the inequality \eqref{deltaineq} we have that $||p_n-p_m||_2\ge1-10\delta>0$ for every $n,m\in\N$ with $n\neq m$. This is impossible because $(p_n)$ converges to $p$. Hence, there is $N\in\N$ such that $(p_n)\subset \bigcup\limits_{n=1}^NO_n$ and therefore $p\in\bigcup\limits_{n=1}^NO_n$. This proves that $O$ is closed and it only remains to prove that $O$ is convex.

Let us take two distinct points $x,y\in O$. If $x,y\in O_n$ for some $n\in\N$ then $\frac{x+y}{2}\in O_n\subset O$ (if $O_n$ is supposed to be strictly convex then $\frac{x+y}{2}\in \text{Int}(O_n)\subset \text{Int}(O)$). Otherwise, there are $n,m\in\N$ with $n\neq m$ such that $x\in O_n\setminus O_m$ and $y\in O_m\setminus O_n$. Hence, $x\in \widetilde B_n\setminus B_{\ell_2}$ and $y\in \widetilde B_m\setminus B_{\ell_2}$. Thus, by Lemma \ref{closed} it follows that $||x-y||_2\ge1-10\delta$. From \eqref{baseequiv} we know that $||x||_2,||y||_2\le1+2\delta$ so that $\widetilde x:=\frac{x}{1+2\delta},\widetilde y:=\frac{y}{1+2\delta}\in B_{\ell_2}$ with $||\widetilde x-\widetilde y||\ge\frac{1-10\delta}{1+2\delta}$. Hence, by the defintion of the modulus $\delta_2$, Fact \ref{modcon} and inequality \eqref{deltaineq},
$$\begin{aligned}\Big|\Big|\frac{ x+y}{2}\Big|\Big|_2=(1+2\delta)\Big|\Big|\frac{\widetilde x+\widetilde y}{2}\Big|\Big|_2\le&(1+2\delta)\bigg(1-\delta_2\Big( \frac{1-10\delta}{1+2\delta} \Big)\bigg)\\=&(1+2\delta)\sqrt{1-\Big(\frac{1-10\delta}{1+2\delta}\Big)^2}<1.
\end{aligned}$$
This means that $\frac{x+y}{2}\in\text{Int}(B_{\ell_2})\subset \text{Int}(O)$.
\end{proof}
\end{prop}

Now, we define for every $n\in\N$ a new norm in $\ell_2$ as $|\cdot|_n:\ell_2\to[0,\infty)$ given by
$$|x|_n=\Big(1-\frac{\mu}{n^2}\Big)||x||_n+\frac{\mu}{n^2}||x||_{2},$$
where $0<\mu\le\frac{\delta^3}{132(1+2\delta)}$.
Clearly, the norm $|\cdot|_n$ is rotund for every $n\in\N$. We denote $B_n=\{x\in \ell_2\;:\;|x|_n\le1\}$ the strictly convex unit ball for the norm $|\cdot|_n$.
From \eqref{baseequiv} we get that
$$\frac{1}{1+2\delta}||x||_2\le||x||_n\le|x|_n\le||x||_2\;\;\;\forall x\in \ell_2,\;\;\;\forall n\in\N.$$
Hence, for every $n\in\N$,
\begin{equation}\label{equivnfine}B_{\ell_2}\subset B_n\subset \widetilde B_n\subset (1+2\delta)B_{\ell_2}\end{equation}
Let us denote for each $n\in\N$ and $i=1,2$ the functional
$$f^*_{i,n}=e^*_{2n}+(-1)^{i+1}\frac{\delta}{2n}e_{2n+1}^*,$$
where $e^*_n(\cdot)=\langle e_n,\cdot\rangle$ for every $n\in\N$. 

\begin{lemma}\label{mainlemmaNPM}
Let $n\in\N$ and $i\in\{1,2\}$. If  $x\in S^i_n:=\big\{y\in \frac{n^2+\mu}{n^2}B_n\;:\;f^*_{i,n}(y)\ge 1+\frac{\delta^2}{2n^2}\big\}$  then
$$||x-z_{i,n}||_n\le\frac{33\mu}{\delta^2}.$$
\end{lemma}
Let us first state and prove the following Claim \ref{3ineq},
\begin{claim}\label{3ineq}
For every $n\in\N$ and $i=1,2$ the following inequality is satisfied,
\begin{equation}\label{ineq1}\sup_{x\in \co((P_n\setminus\{z_{i,n}\})\cup B_{\ell_2})} f^*_{i,n}(x)<1+\frac{\delta^2}{4n^2},\end{equation}
\begin{proof}[Proof of Claim \ref{3ineq}]
We compute $\sup f^*_{i,n}\big(B_{\ell_2}\big)$. If $x=(x_n)\in B_{\ell_2}$ then
$$f^*_{i,n}(x)=x_{2n}+(-1)^{i+1}\frac{\delta}{2n}x_{2n+1}\le \sqrt{1-x_{2n+1}^2}+\frac{\delta}{2n}|x_{2n+1}|.$$
Hence,
$$\begin{aligned}\sup f^*_{i,n}\big(B_{\ell_2}\big)\le&\max\limits_{t\in[0,1]}\sqrt{1-t^2}+\frac{\delta}{2n}t\\=&\sqrt{1-\frac{\Big(\frac{\delta}{2n}\Big)^2}{1+\Big(\frac{\delta}{2n}\Big)^2}}+\frac{\Big(\frac{\delta}{2n}\Big)^2}{\sqrt{1+\Big(\frac{\delta}{2n}\Big)^2}}<1+\Big(\frac{\delta}{2n}\Big)^2.\end{aligned}$$
Finally, if $j\in\{1,2\}\setminus \{i\}$ then
$$\begin{aligned}\sup_{x\in \co((P_n\setminus\{z_{i,n}\})\cup B_{\ell_2})} f^*_{i,n}(x)=&\sup_{x\in (P_n\setminus\{z_{i,n}\})\cup B_{\ell_2}} f^*_{i,n}(x)\\=&\max\{\sup f^*_{i,n}\big(P_n\setminus\{z_{i,n}\}\big),\sup f^*_{i,n}\big(B_{\ell_2}\big)\}\\=&\max\{f^*_{i,n}(z_{j,n}),\sup f^*_{i,n}\big(B_{\ell_2}\big)\}<1+\frac{\delta^2}{4n^2},\end{aligned}$$
as desired.
\end{proof}
\end{claim}
\begin{proof}[Proof of Lemma \ref{mainlemmaNPM}]
Let us take $i\in\{1,2\}$ and $n\in\N$ and define the closed slice
$$\widetilde S^i_n=\bigg\{y\in \widetilde B_n\;:\;f^*_{i,n}(y)\ge\frac{n^2}{n^2+\mu}\Big(1+\frac{\delta^2}{2n^2}\Big)\bigg\}.$$
If $y\in \co\big(P_n\cup B_{\ell_2}\big)$ then there is $\lambda\in [0,1]$ and $p\in \co\big((P_n\setminus\{z_{i,n}\})\cup B_{\ell_2}\big) $ such that
$$y=\lambda p+(1-\lambda)z_{i,n}.$$
If we assume that $y\in \widetilde S^i_n$ then by Claim \ref{3ineq},
$$\begin{aligned}\frac{n^2}{n^2+\mu}\Big( 1+\frac{\delta^2}{2n^2} \Big)\le f^*_{i,n}(y)=&\lambda f^*_{i,n}(p)+(1-\lambda)f^*_{i,n}(z_{i,n})\\\le& \lambda\Big(1+\frac{\delta^2}{4n^2}\Big)+(1-\lambda)\Big(1+\frac{\delta^2}{2n^2}\Big).\end{aligned}$$
Hence, $\lambda\le \frac{\mu 4n^2}{\delta^2(n^2+\mu)}\big(1+\frac{\delta^2}{2n^2}\big)<\frac{8\mu}{\delta^2}$ and so
$$||y-z_{i,n}||_n=||\lambda p-\lambda z_{i,n})||_n\le 2\lambda<\frac{16\mu}{\delta^2}.$$
Therefore, we have proven that
\begin{equation}\label{firstineq}
||y-z_{i,n}||_n\le\frac{16\mu}{\delta^2}\;\;\;\;\;\;\forall y\in \widetilde S^i_n.
\end{equation}
Finally, let us take $x\in S^i_n$. Then we get that $y:= \frac{n^2}{n^2+\mu}x\in \widetilde S^i_n$. Hence, by \eqref{firstineq} and the triangle inequality,
$$\begin{aligned}|| x-z_{i,n} ||_n\le& \Big|\Big| x-\frac{n^2+\mu}{n^2}z_{i,n} \Big|\Big|_n+\Big|\Big| \frac{n^2+\mu}{n^2}z_{i,n}-z_{i,n} \Big|\Big|_n\\\le&\frac{n^2+\mu}{n^2}||y-z_{i,n} ||_n+\frac{\mu}{n^2}<\frac{(n^2+\mu)16\mu}{n^2\delta^2}+\frac{\mu}{n^2}<\frac{33\mu}{\delta^2}.\end{aligned}$$
\end{proof}

Also, from \eqref{equivnfine} we know that $B_{\ell_2}\subset B_n\subset \widetilde B_n$. Hence, by Proposition \ref{indexB} we have that
$$B=\bigcup\limits_{n\in\N}B_n,$$
is closed, symmetric and strictly convex. Moreover, by \eqref{equivnfine}, we have that
$$B_{\ell_2}\subset B\subset(1+2\delta)B_{\ell_2}.$$
Hence, if $||\cdot||$ is the norm given by $B$, clearly $||\cdot||$ is a rotund norm in $\ell_2$ such that 
\begin{equation}\label{dBM}d_{BM}\big((\ell_2,||\cdot||),(\ell_2,||\cdot||_2)\big)\le\ln(1+2\delta)\le2\delta.\end{equation}

Given a set $K\subset \ell_2$ we will say that $K$ is finitely supported if there is $N\in\N$ such that $K\subset [e_i]_{i=1}^N$.

\begin{theorem}\label{NPMtheo}
Let $K\subset \ell_2$ be a nontrivial finitely supported compact convex set such that $\emph{diam}(K\cap[e_1])>0$. Then, there is $p\in K$ such that the nearest point map from $\ell_2$ onto $K$ given by the norm $||\cdot||$ is not uniformly continuous even when restricted to $p+\ep B_{\ell_2}$ for any $\ep>0$.
\begin{proof}
Given $n\in\N$ if we take $m\in\N$ with $n\neq m$ then $B_m\subset \widetilde B_m$ and $f_{i,n}^*\big(P_m\big)=\{0\}$ for $i=1,2$. Hence, by \eqref{ineq1},
$$\sup f^*_{i,n}\big( B_m\big)\le\sup f^*_{i,n}\big(\widetilde B_m\big)=\sup f^*_{i,n}\big(B_{\ell_2}\big)<1+\frac{\delta^2}{4n^2}.$$
Since $m\neq n$ was arbitrary and $B=\bigcup_nB_n$ we deduce that
\begin{equation}\label{mn}\sup f^*_{i,n}\big(B\setminus B_n\big)<1+\frac{\delta^2}{4n^2}.\end{equation}
Let us take $N\in\N$ such that $K\subset [e_i]_{i=1}^N$. We may assume that $\pm \ep\delta e_1\in K$ (translating $K$ if necessary and taking $\ep>0$ small enough). Let us denote $R:\ell_2\to K$ the nearest point map for the norm $||\cdot||$. We claim that
\begin{equation}\label{goalclaim}||R(\ep x_{i,n})+(-1)^i\ep\delta e_1||\le \ep\frac{66\mu}{\delta^2}\;\;\;\;\forall n>N\;,\;\;\forall i=1,2.\end{equation}
Indeed, by definition of $R$ we have that
$$\begin{aligned}||\ep x_{i,n}-R(\ep x_{i,n})||\le&||\ep x_{i,n}+(-1)^i\ep\delta e_1||=\ep||z_{i,n}||\le\ep|z_{i,n}|_n\\=&\ep\Big(1-\frac{\mu}{n^2}\Big)||z_{i,n}||_n+\ep\frac{\mu}{n^2}||z_{i,n}||_2\le\ep\frac{n^2+\mu}{n^2}.\end{aligned}$$
Equivalently, $\ep x_{i,n}-R(\ep x_{i,n})\in \ep\frac{n^2+\mu}{n^2}B$. Since $\mu<\frac{\delta^2}{5}$, using \eqref{mn} we get that for every $n>N$,
$$\begin{aligned}f^*_{i,n}\big(\ep x_{i,n}-R(\ep x_{i,n})\big)=&\ep f^*_{i,n}(x_{i,n})=\ep\Big( 1+\frac{\delta^2}{2n^2} \Big)\\>&\ep\frac{n^2+\mu}{n^2}\Big(1+\frac{\delta^2}{4n^2}\Big)\ge \sup f^*_{i,n}\Big(\ep\frac{n^2+\mu}{n^2}B\setminus\ep\frac{n^2+\mu}{n^2}B_n \Big).\end{aligned}$$
Hence, $\ep x_{i,n}-R(\ep x_{i,n})\in \ep \frac{n^2+\mu}{n^2}B_n$. Since we have shown that $f^*_{i,n}\big(\ep x_{i,n}-R(\ep x_{i,n})\big)=\ep\big( 1+\frac{\delta^2}{2n^2} \big)$ we conclude that $\ep x_{i,n}-R(\ep x_{i,n})\in \ep S^i_n$. By Lemma \ref{mainlemmaNPM} we have that $||\ep x_{i,n}-R(\ep x_{i,n})-\ep z_{i,n}||_n\le\ep\frac{33\mu}{\delta^2}$. Thus, we finally prove \eqref{goalclaim} since
$$\begin{aligned}||R(\ep x_{i,n})+(-1)^i\ep\delta e_1||\le&(1+2\delta)||R(\ep x_{i,n})+(-1)^i\ep\delta e_1||_n\\=&(1+2\delta)||\ep x_{i,n}-R(\ep x_{i,n})-\ep z_{i,n}||_n\le\ep(1+2\delta)\frac{33\mu}{\delta^2}.\end{aligned}$$
Now we are able to prove the statement of the theorem. In fact, since $\mu\le\frac{\delta^3}{132(1+2\delta)}$ then by \eqref{goalclaim} and the triangle inequality,
$$\begin{aligned}||R(\ep x_{1,n})-R(\ep x_{2,n})||\ge&||2\ep\delta e_1||-\sum\limits_{i=1}^2||R(\ep x_{i,n})+(-1)^i\ep\delta e_1||\\\ge&\frac{2\ep\delta}{1+2\delta}-\ep\frac{132\mu}{\delta^2}\ge\frac{\ep\delta}{1+\delta}>0,\;\;\forall n>N.\end{aligned}$$
We are now done since $||\ep x_{1,n}-\ep x_{2,n}||=\frac{2\ep\delta}{n}||e_{2n+1}||\xrightarrow{n\to\infty}0$.
\end{proof}
\end{theorem}

Finally, the proof of Theorem \ref{mainNPM} follows now easily from Theorem \ref{NPMtheo}.

\begin{proof}[Proof of Theorem \ref{mainNPM}]
Let $K$ be a nontrivial finite dimensional compact convex subset of $\ell_2$. There is $N\in\N$ such that $\text{dim}([K])=N$ and there is $v\in [K]$ with $||v||_2=1$ such that $\text{diam}(K\cap[v])>0$. Then, due to the symmetries of $\ell_2$ and the orthogonal projection theorem, there is an onto isometry $T:\ell_2\to\ell_2$ such that $T([K])=[e_i]_{i=1}^N$ and $T(v)=e_1$. Hence, the nontrivial compact convex set $\widetilde K=T(K)$ is a finitely supported subset of $\ell_2$ such that $\text{diam}(\widetilde K\cap [e_1])>0$. Therefore, by Theorem \ref{NPMtheo} we know that the nearest point map from $\ep B_{\ell_2}$ to $\widetilde K$ given by the norm $||\cdot||$ is not uniformly continuous. We are done just taking the norm $|\cdot|$ given by
$$|x|=||T(x)||\;\;\;\;\forall x\in\ell_2.$$
\end{proof}

Theorem \ref{mainNPM} strongly contrasts with the fact that finite dimensional compact convex sets are clearly absolute Lipschitz retracts.

\bigskip

\textbf{Acknowledgements.} I would like to thank Petr Hájek for his valuable advice and support during the development of this note. I  am also grateful to Gilles Godefroy for encouraging me to carry out this project.

\bigskip

\printbibliography
\end{document}